\documentclass[12pt]{article}

\usepackage[leqno]{amsmath}
\usepackage{amssymb}
\usepackage{amsthm}
\usepackage{nicematrix}
\usepackage{anyfontsize}
\usepackage{hyperref}

\newtheorem{theorem}{Theorem}[section]
\newtheorem{corollary}[theorem]{Corollary}
\newtheorem{proposition}[theorem]{Proposition}
\newtheorem{lemma}[theorem]{Lemma}

\newtheorem*{theorem*}{Theorem}

\theoremstyle{remark}
\newtheorem{remark}[theorem]{Remark}
\theoremstyle{definition}

\newtheorem*{definition*}{Definition}

\numberwithin{equation}{section}

\newcommand{\C}{\mathbb{C}}
\newcommand{\Z}{\mathbb{Z}}

\newcommand{\ox}{\otimes}
\DeclareMathOperator{\Pf}{Pf}
\DeclareMathOperator{\ind}{ind}
\DeclareMathOperator{\sgn}{sgn}

\newcommand{\Neg}{\phantom{-}} 
\newcommand{\dsp}{\displaystyle}

\newcommand{\bibDate}[1]{}

\title{Pfaffian Formulation of Schur's $Q$-functions}
\author{John~Graf, Naihuan~Jing}
\date{\today} 

\begin{document}

    \maketitle
    
    \begin{abstract}
        We introduce a Pfaffian formula that extends Schur's $Q$-functions $Q_\lambda$ to be indexed by compositions $\lambda$ with negative parts. This formula makes the Pfaffian construction more consistent with other constructions, such as the Young tableau and Vertex Operator constructions. With this construction, we develop a proof technique involving decomposing $Q_\lambda$ into sums indexed by partitions with removed parts. Consequently, we are able to prove several identities of Schur's $Q$-functions using only simple algebraic methods.
    \end{abstract}

    
    \tableofcontents

    \section{Introduction}

    Schur's $Q$-functions were introduced in 1911 by Schur \cite{schur:1911} to study projective characters of the symmetric and alternating groups. Independent Schur's $Q$-functions $Q_{\lambda}$ are indexed by strict partitions $\lambda$, and it is customary to extend the definition of $Q_{\lambda}$ to compositions using the anti-symmetry property $Q_{(r,s)}=-Q_{(s,r)}$, for $r>s\geq0$. Consequently, for any partition $\lambda$, the function $Q_\lambda$ can be defined as the Pfaffian of the skew-symmetric matrix $(Q_{(\lambda_i,\lambda_j)})$, which can be viewed as a spin analog of the Jacobi-Trubi formula for the Schur functions \cite{macdonald:1995}. In 1990 and 1991, Nimmo \cite{nimmo:1990} and J{\'o}zefiak and Pragacz \cite{jozefiak-pragacz:1991} introduced a Pfaffian formula for skew Schur's $Q$-functions. Other constructions of Schur's $Q$-functions involve Young tableau \cite[p.~229]{macdonald:1995}, shifted Young tableaux \cite{stembridge:1989}, vertex operators \cite{jing:1991a}, and generating functions \cite[p.~253]{macdonald:1995}. 
    
    However, the above anti-symmetry property does not hold for negative integers $r,s<0$, so the matrix is not skew-symmetric if we allow $\lambda$ to have negative parts. In fact, one must also define $Q_{(0,0)}$ to be $0$ for this matrix to be skew-symmetric, even when only considering positive parts (see \cite[Theorem~4.1]{okada:2019}). This convention is not consistent with the definition that $q_0=1$, and the combinatorial fact that $\lambda$ and $\lambda0=(\lambda_1,\ldots,\lambda_n,0)$ have the same Young diagram. So, instead we should have $Q_{(0,0)}=q_0=1$, as is the case in other constructions of Schur's $Q$-functions. One may reconcile this by simply defining the diagonal of the matrix to be all $0$s \cite{huang-chu-li:2023}.
    
    In this paper, we resolve these issues by defining $Q_\lambda$ with a slightly different matrix that is always skew-symmetric, even when $\lambda$ has negative parts. Schur's $Q$-functions $Q_{(r,s)}$ with two, possibly negative, parts have been considered before \cite{ogawa:2009}, and we extend this definition to a composition $\lambda$ of any length. As a result, we are able to algebraically prove many results regarding Schur's $Q$-functions. In particular, we show that if $\lambda$ is a composition containing negative parts $-p_1,\ldots,-p_k$, then $Q_\lambda$ is, up to a coefficient, equal to the function $Q_\mu$, where $\mu$ is the composition obtained from $\lambda$ by removing the parts $\pm p_1,\ldots,\pm p_k$.

    Furthermore, we develop a useful technique of decomposing Schur's $Q$-functions into sums involving the functions $Q_{\lambda\setminus\{\lambda_i\}}$, where $\lambda\setminus\{\lambda_i\}$ is the partition obtained by removing the $i$th part of $\lambda$. Using this technique, we are able to algebraically prove fundamental identities such as $Q_{\lambda0}=Q_\lambda$ and $Q_{\lambda/0}=Q_\lambda$. These identities are not true in the traditional Pfaffian formulation where we require $Q_{(0,0)}$ to be $0$, but they are simple identities in other constructions. Crucially, we are now able to use these important identities with Pfaffian calculations.
    
    Our main theorem is a Pfaffian formulation of a vertex operator identity, which we prove using simple algebraic methods. As a consequence, we are able to decompose $Q_\lambda$ into different sums of skew functions. We are also able to find a connection between the functions $Q_{\lambda\setminus\{\lambda_k\}}$ and the $A_n$ root system. Finally, we use our techniques to prove some additional identities of Schur's $Q$-functions.
    
    \section{Preliminaries}
    
    
    A general reference for symmetric functions and Schur's $Q$-functions is \cite{macdonald:1995}. A reference for the standard Pfaffian formulas for Schur's $Q$-functions is \cite{jozefiak-pragacz:1991}.
    
    \subsection{Compositions and Partitions}
    
    A \emph{composition} $\lambda=(\lambda_1,\ldots,\lambda_n)\in\Z^n$ is a finite sequence of integers $\lambda_i$, called its \emph{parts}. A \emph{partition} is a weakly decreasing composition with nonnegative parts. For any composition $\lambda$, its \emph{weight} $|\lambda|$ is the sum of its parts, and its \emph{length} $\ell(\lambda)$ is the number of nonzero parts. We say a composition is \emph{strict} if each nonzero part is distinct. 
    We will denote the set of partitions (resp. strict partitions) by $\mathcal P$ (resp. $\mathcal{SP}$). For any integer $p\in\Z$, we define
        \[p\lambda:=(p,\lambda_1,\ldots,\lambda_n).\]
    Additionally, we can append a $0$ as a final part,
    \[\lambda0:=(\lambda_1,\ldots,\lambda_n,0).\]
    For any $i\in\{1,2,\ldots,n\}$, we define
    \[\lambda\setminus\{\lambda_i\}:=(\lambda_1,\ldots,\lambda_{i-1},\lambda_{i+1},\ldots,\lambda_n).\]
    We may similarly remove multiple parts, $\lambda\setminus\{\lambda_i,\lambda_j\}$. Note that this notation specifies the index of the part being removed from $\lambda$. For example, if $\lambda=(5,3,3,1)$ then we would write $\lambda\setminus\{\lambda_2\}=(5,3,1)$, but not $(5,3,3,1)\setminus\{3\}$. Additionally, if we remove all of the parts, then we are left with the $0$ composition.
    
    Finally, for strict compositions $\lambda\in\Z^n$ and any integer $p\in\Z$, we define the \emph{index of $p$ in $\lambda$} to be
    \[\ind(\lambda,p):=
        \begin{cases}
            i&\text{if }p=\lambda_i\text{ for some }i,\\
            0&\text{otherwise}.
        \end{cases}
    \]
    
    \subsection{Pfaffians}
        
    Suppose $M=(m_{ij})_{2n\times2n}$ is a skew-symmetric matrix. Then the determinant of $M$ is a square, and we define the \emph{Pfaffian} $\Pf M$ of $M$ by the equation $(\Pf M)^2=\det M$. We recount a formula from \cite{cao-jing-liu:2024} to compute a Laplace-type expansion of the Pfaffian. Fix some $i\in\{1,2,\ldots,2n\}$, then
    \[\Pf M=(-1)^{i-1}\sum_{j\neq i}(-1)^j:m_{ij}:\Pf M_{ij},\]
    where
    \[:m_{ij}:~:=
        \begin{cases}
            m_{ij}&\text{if }i<j,\\
            m_{ji}&\text{if }i>j,
        \end{cases}\]
    and $M_{ij}$ is the submatrix obtained from $M$ by deleting the $i$th, $j$th rows and the $i$th, $j$th columns.
    
    \subsection{Schur's $Q$-functions}
    
    Now, we may define Schur's $Q$-functions in the variables $x_1,x_2,\ldots$. We will omit the variables and write $q_n$, $Q_\lambda$, etc., in place of $q_n(x_1,x_2,\ldots)$, $Q_\lambda(x_1,x_2,\ldots,)$, etc. The functions $q_n$ are defined by the generating series
    \[\kappa_z:=\prod_{i=1}^\infty\frac{1+zx_i}{1-zx_i}=\sum_{n\in\Z}q_nz^n.\]
    Note that we have $q_0=1$, and $q_n=0$ for $n<0$. For any integers $r,s\in\Z$, we define
    \begin{equation}\label{Q_(r,s)}
        Q_{(r,s)}:=q_rq_s+2\sum_{i=1}^s(-1)^iq_{r+i}q_{s-i},
    \end{equation}
    where the sum is empty when $s<1$. In particular, we have $Q_{(r,0)}=q_r$ for all $r\in\Z$.
    
    This formula is the standard definition of Schur's $Q$-function for a partition $(r,s)$ with two parts, except we allow $r$ and $s$ to be negative integers. So, we have the antisymmetry property $Q_{(r,s)}=-Q_{(s,r)}$ for all $r,s\in\Z$ except when $s=-r$. In particular, for $r,s>0$ we have
    \begin{equation}\label{Q_(r,-r)}
        \begin{split}
            Q_{(\pm r,-s)}&=0,\\
            Q_{(-r,s)}&=0\quad(r\neq s),\\
            Q_{(-r,r)}&=(-1)^r2,\\
            Q_{(0,0)}&=1.
        \end{split}
    \end{equation}
    Note that our convention is consistent with the vertex operator construction that $Q_{(-r,r)}=(-1)^r2$ and $Q_{(r,-r)}=0$ for $r>0$. We define \emph{Schur's $Q$-function} for any composition $\lambda\in\Z^n$ to be 
    \[Q_\lambda:=
        \begin{cases}
        \Pf M(\lambda)&\text{if }n\text{ is even},\\
        \Pf M(\lambda0)&\text{if }n\text{ is odd},
        \end{cases}\]
    where the $ij$-entry of the $n\times n$ matrix $M(\nu)$ (or $(n+1)\times(n+1)$ if $n$ is odd) is
    \[M(\nu)_{ij}:=
        \begin{cases}
            Q_{(\nu_i,\nu_j)}(A)&\text{if }i>j,\\
            0&\text{if }i=j,\\
            -Q_{(\nu_j,\nu_i)}(A)&\text{if }i< j.
        \end{cases}\]
    
    Next, we define the \emph{skew Schur's $Q$-function} for compositions $\lambda\in\Z^n$ and $\mu\in\Z^m$ to be
    \[Q_{\lambda/\mu}:=
        \begin{cases}
            \Pf M(\lambda,\mu)&\text{if }n+m\text{ is even},\\
            \Pf M(\lambda0,\mu)&\text{if }n+m\text{ is odd},
        \end{cases}\]
    where $M(\lambda,\mu)$ denotes the matrix
    \[M(\lambda,\mu):=
        \begin{pmatrix}
            M(\lambda)&N(\lambda,\mu)\\
            -N(\lambda,\mu)^t&0
        \end{pmatrix},\]
    and where $N(\lambda,\mu)$ is the $n\times m$ matrix
    \[N(\lambda,\mu):=
        \begin{pmatrix}
            q_{\lambda_1-\mu_m}(A)&\cdots&q_{\lambda_1-\mu_1}(A)\\
            \vdots&&\vdots\\
            q_{\lambda_n-\mu_m}(A)&\cdots&q_{\lambda_n-\mu_1}(A)
        \end{pmatrix}.\]

    Let $\Gamma:=\Z[q_1,q_2,\ldots]$ be the subring of symmetric functions generated by the $q_i$ over integers. Clearly $Q_{\lambda}\in \Gamma$, and we will see that $Q_{\lambda}$, $\lambda\in\mathcal{SP}$, form a spanning set, and in fact basis, of $\Gamma$.
    
    \section{Fundamental Properties of $Q_\lambda$}

    When $\Gamma$ is considered as a $\Z$-module, the functions $Q_\nu$ form a basis of $\Gamma$, for strict partitions $\nu$ \cite{macdonald:1995}. Hence, it is desirable to express $Q_\lambda$ in this basis for any composition $\lambda$. However, in this section we will show that it can be convenient to express Schur's $Q$-functions with the form $Q_{p\lambda}$, where $p\in\Z$ is any integer, and $\lambda$ is a strict partition. We will also discuss some basic properties of the Pfaffian realization of Schur's $Q$-functions that are necessary to construct more complicated identities.
    
    \subsection{Accounting for the Disparity of $\ell(\lambda)$}
    
    Due to the Pfaffian definition of $Q_\lambda$, we need to treat $\lambda$ differently depending on the parity of its length $\ell(\lambda)$. Indeed, many subsequent identities will have different formulas depending on the length of $\lambda$. The following proposition provides an explanation for these differences.
    
    \begin{proposition}\label{alternating sum identity}
        Let $\lambda\in\Z^n$ be a composition, then
        \[-\sum_{i=1}^n(-1)^iq_{\lambda_i}Q_{\lambda\setminus\{\lambda_i\}}=
            \begin{cases}
                Q_\lambda&\text{if }n\text{ is odd},\\
                0&\text{if }n\text{ is even}.
            \end{cases}\]
    \end{proposition}
    
    \begin{proof}
        If $n$ is odd, then we have
        \begin{align*}
            Q_{\lambda}&=\Pf\begin{pNiceArray}{ccc|c}
                \Block{3-3}<\fontsize{30}{50}>{M(\lambda)} &&& Q_{(\lambda_1,0)}\\
                &&& \vdots\\
                &&& Q_{(\lambda_n,0)}\\\hline
                -Q_{(\lambda_1,0)} & \cdots & -Q_{(\lambda_n,0)} & 0
            \end{pNiceArray}\\
            &=(-1)^{n+1-1}\sum_{j=1}^{\ell(\lambda)}(-1)^jq_{\lambda_j}\Pf M(\lambda\setminus\{\lambda_j\})
        \end{align*}
        by expanding along the last row/column, and since $Q_{(\lambda_j,0)}=q_{\lambda_j}$. Then, note that $\Pf(M(\lambda\setminus\{\lambda_j\}))=Q_{\lambda\setminus\{\lambda_j\}}$.
        
        If $n$ is even, then we have that $Q_{\lambda\setminus\{\lambda_i\}}=\Pf M((\lambda\setminus\{\lambda_i\})0)$, and so $Q_{\lambda\setminus\{\lambda_i\}}$ is the Pfaffian of the matrix
        \[\begin{pNiceArray}{cccccc|c}
            \Block{6-6}<\fontsize{30}{50}>{M(\lambda\setminus\{\lambda_i\})} &&&&&& Q_{(\lambda_1,0)}\\
            &&&&&& \vdots\\
            &&&&&& Q_{(\lambda_{i-1},0)}\\
            &&&&&& Q_{(\lambda_{i+1},0)}\\
            &&&&&& \vdots\\
            &&&&&& Q_{(\lambda_n,0)}\\\hline
            -Q_{(\lambda_1,0)} & \cdots & -Q_{(\lambda_{i-1},0)} & -Q_{(\lambda_{i+1},0)} & \cdots & -Q_{(\lambda_n,0)} & 0
        \end{pNiceArray}.\]
        Expanding on the last row/column, we see that we have 
        \begin{align*}
            Q_{\lambda\setminus\{\lambda_i\}}&=(-1)^{n-1}\left(\sum_{j=1}^{i-1}(-1)^jq_{\lambda_j}Q_{\lambda\setminus\{\lambda_i,\lambda_j\}}-\sum_{j=i+1}^n(-1)^jq_{\lambda_j}Q_{\lambda\setminus\{\lambda_i,\lambda_j\}}\right)\\
            &=\sum_{\substack{1\leq j\leq n\\j\neq i}}\sgn(i,j)(-1)^jq_{\lambda_j}Q_{\lambda\setminus\{\lambda_i,\lambda_j\}}
        \end{align*}
        where $\sgn(i,j)=
            \begin{cases}
                -1  & \text{if }j<i,\\
                1 & \text{if }j>i.
            \end{cases}$
            
        Finally, we substitute this formula for $Q_{\lambda\setminus\{\lambda_i\}}$ into the sum to get
        \begin{align*}
            \sum_{i=1}^n(-1)^iq_{\lambda_i}Q_{\lambda\setminus\{\lambda_i\}}&=\sum_{i=1}^n(-1)^iq_{\lambda_i}\sum_{\substack{1\leq j\leq n\\j\neq i}}\sgn(i,j)(-1)^jq_{\lambda_j}Q_{\lambda\setminus\{\lambda_i,\lambda_j\}}\\
            &=\sum_{i=1}^n\sum_{\substack{1\leq j\leq n\\j\neq i}}(-1)^{i+j}\sgn(i,j)q_{\lambda_i}q_{\lambda_j}Q_{\lambda\setminus\{\lambda_i,\lambda_j\}}\\
            &=\sum_{1\leq i<j\leq n}(-1)^{i+j}(q_{\lambda_i}q_{\lambda_j}-q_{\lambda_j}q_{\lambda_i})Q_{\lambda\setminus\{\lambda_i,\lambda_j\}}\\
            &=0.
        \end{align*}
    \end{proof}

    As a consequence of this identity, we are able to prove the following fundamental property of Schur's $Q$-functions.
    
    \begin{proposition}\label{Q_lambda0}
        Let $\lambda\in\Z^n$ be a composition, then $Q_{\lambda0}=Q_\lambda$.
    \end{proposition}
    
    \begin{proof}
        If $n$ is odd, then this statement is simply the definition of $Q_\lambda$. If $n$ is even, then by definition we have $Q_{\lambda0}=\Pf M(\lambda00)$, and so
        \begin{align*}
            Q_{\lambda0}&=
                \Pf\begin{pNiceArray}{ccc|cc}
                    \Block{3-3}<\fontsize{30}{50}>{M(\lambda)} &&& Q_{(\lambda_1,0)} & Q_{(\lambda_1,0)}\\
                    &&& \vdots & \vdots\\
                    &&& Q_{(\lambda_n,0)} & Q_{(\lambda_n,0)}\\\hline
                    -Q_{(\lambda_1,0)} & \cdots & -Q_{(\lambda_n,0)} & 0 & Q_{(0,0)}\\
                    -Q_{(\lambda_1,0)} & \cdots & -Q_{(\lambda_n,0)} & -Q_{(0,0)} & 0
                \end{pNiceArray}\\
                &=
                \Pf\begin{pNiceArray}{ccc|cc}
                    \Block{3-3}<\fontsize{30}{50}>{M(\lambda)} &&& q_{\lambda_1} & q_{\lambda_1}\\
                    &&& \vdots & \vdots\\
                    &&& q_{\lambda_n} & q_{\lambda_n}\\\hline
                    -q_{\lambda_1} & \cdots & -q_{\lambda_n} & 0 & 1\\
                    -q_{\lambda_1} & \cdots & -q_{\lambda_n} & -1 & 0
                \end{pNiceArray}.
        \end{align*}
        By expanding along the last row/column, we have
        \begin{align*}
            Q_{\lambda0}&=(-1)^{n+2-1}\sum_{i=1}^n(-1)^iq_{\lambda_i}\Pf(M(\lambda\setminus\{\lambda_i\}0))\\
            &\quad+(-1)^{n+2-1}(-1)^{n+1}1\Pf(M(\lambda))\\
            &=-\sum_{i=1}^n(-1)^iq_{\lambda_i}Q_{\lambda\setminus\{\lambda_i\}}+Q_\lambda
        \end{align*}
       By Proposition \ref{alternating sum identity} the sum is $0$ when $n$ is even, and hence we are left with $Q_{\lambda0}=Q_\lambda$.
    \end{proof}
    
    In other words, we may append an arbitrarily-long, finite sequence of $0$s to $\lambda$ without changing $Q_\lambda$. This is consistent with other constructions, whereas the traditional Pfaffian definition of Schur's $Q$-functions implies that doing so gives $0$. In particular, the typical Pfaffian construction defines $Q_{(0,0)}$ to be $0$, which is not consistent with the convention that $q_0=1$, whereas here we use (\ref{Q_(r,s)}) to define $Q_{(0,0)}=1$ more naturally. We remark that Proposition \ref{Q_lambda0} implicitly implies that $q_0=1$. 

    However, it is worth noting that our Pfaffian definitions are otherwise equivalent to the typical Pfaffian definitions of $Q_\lambda$ and $Q_{\lambda/\mu}$. Note that $q_{\lambda_i-0}=Q_{(\lambda_i,0)}$ and that $q_{0-\mu_i}=0$ (when $\mu_i>0$), so we also have $Q_{\lambda/\mu}=\Pf M(\lambda,\mu0)$ when $n+m$ is odd (also see \cite{okada:2019} and \cite[Theorem 1.6]{hamel:1996}). Therefore, if we only consider strict partitions $\lambda$ and $\mu$ and we set $Q_{(0,0)}$ to be $0$, then we get the same formulas for $Q_\lambda$ and $Q_{\lambda/\mu}$ as in \cite{jozefiak-pragacz:1991}.
    
    As a result of Proposition \ref{Q_lambda0}, for compositions $\lambda\in\Z^n$ we may sometimes assume $n$ is even or odd, as convenient (in general, the length $\ell(\lambda)$ may have different parity since $\ell(\lambda)\leq n$). We identify compositions that differ by any finite number of trailing $0$s for the purpose of computing $Q_\lambda$. It is clear that we may also append arbitrarily many zeros to $\lambda$ when computing $Q_{\lambda/\mu}$, although we may not always append zeros to $\mu$. 
    
    Next, we have another fundamental identity that is not true in the traditional Pfaffian construction.
    
    \begin{corollary}
        Let $\lambda\in\Z^n$ be a composition, then $Q_{\lambda/0}=Q_\lambda$.
    \end{corollary}
    
    \begin{proof}
        This is immediate since we see that $M(\lambda,0)=M(\lambda0)$ if $n$ is odd, and $M(\lambda0,0)=M(\lambda00)$ if $n$ is even.
    \end{proof}
    
    \subsection{Reordering the Parts of $\lambda$}

    For any composition $\lambda\in\Z^n$ and integer $i\in\{1,2,\ldots,n-1\}$, let $B_i$ act on $\lambda$ by swapping parts $\lambda_i$ and $\lambda_{i+1}$,
    \[B_i\lambda:=(\lambda_1,\ldots,\lambda_{i-1},\lambda_{i+1},\lambda_i,\lambda_{i+2},\ldots,\lambda_n).\]
        
    \begin{proposition}\label{B_iQ_lambda}
        Let $\lambda\in\Z^n$ be a composition, then
        \[Q_{B_i\lambda}=
        \begin{cases}
            Q_\lambda&\text{if }\lambda_i=\lambda_{i+1},\\
            -Q_\lambda&\text{if }\lambda_i+\lambda_{i+1}\neq0,\\
            (-1)^{\lambda_i}2Q_{\lambda\setminus\{\lambda_i,\lambda_{i+1}\}}&\text{if }\lambda_i+\lambda_{i+1}=0\text{ and }\lambda_i>0,\\
            0&\text{if }\lambda_i+\lambda_{i+1}=0\text{ and }\lambda_i<0.
        \end{cases}\]
    \end{proposition}
    
    \begin{proof}
        In the first case, we see that $\lambda$ is invariant under $B_i$ if $\lambda_i=\lambda_{i+1}$. Next, note that swapping $\lambda_i$ and $\lambda_{i+1}$ corresponds to interchanging row $i$ with row $i+1$, and column $i$ with column $i+1$, of the matrix $M(\lambda)$ only when $\lambda_i+\lambda_{i+1}\neq0$ due to (\ref{Q_(r,-r)}). In this case, doing so changes the sign of the Pfaffian.
        
        By Proposition \ref{Q_lambda0}, we may assume $n$ is even. So, suppose $\lambda_i+\lambda_{i+1}=0$ and $\lambda_i>0$. Then we have $\lambda_{i+1}=-\lambda_i<0$, and so we see that $Q_\lambda$ is the Pfaffian of the matrix
        \[\begin{pNiceArray}{ccc|cc|ccc}
            \Block{3-3}<\fontsize{20}{50}>{M_1} &&& Q_{(\lambda_1,\lambda_i)} & 0 & \Block{3-3}<\fontsize{20}{50}>{M_2} && \\
             &&& \vdots & \vdots &&& \\
             &&& Q_{(\lambda_{i-1},\lambda_i)} & 0 &&& \\\hline
            -Q_{(\lambda_1,\lambda_i)} & \cdots & -Q_{(\lambda_{i-1},\lambda_i)} & 0 & 0 & Q_{(\lambda_i,\lambda_{i+2})} & \cdots & Q_{(\lambda_i,\lambda_n)}\\
            0 & \cdots & 0 & 0 & 0 & 0 & \cdots & 0 \\\hline
            \Block{3-3}<\fontsize{20}{50}>{M_3} &&& -Q_{(\lambda_i,\lambda_{i+2})} & 0 & \Block{3-3}<\fontsize{20}{50}>{M_4} && \\
             &&& \vdots & \vdots &&& \\
             &&& -Q_{(\lambda_i,\lambda_n)} & 0 &&&
        \end{pNiceArray},\]
        where
        \[
        \begin{pmatrix}
            M_1 & M_2\\
            M_3 & M_4
        \end{pmatrix}=M(\lambda\setminus\{\lambda_i,\lambda_{i+1}\}).\]
        Thus, we see that $Q_\lambda=0$. However, we also see that $Q_{B_i\lambda}$ is the Pfaffian of the matrix
        \[\footnotesize{\begin{pNiceArray}{ccc|cc|ccc}
                \Block{3-3}<\fontsize{20}{50}>{M_1} &&& 0 & Q_{(\lambda_1,\lambda_{i})} & \Block{3-3}<\fontsize{20}{50}>{M_2} && \\
                 &&& \vdots & \vdots &&& \\
                 &&& 0 & Q_{(\lambda_{i-1},\lambda_{i})} &&& \\\hline
                0 & \cdots & 0 & 0 & (-1)^{\lambda_i}2 & 0 & \cdots & 0\\
                -Q_{(\lambda_1,\lambda_{i})} & \cdots & -Q_{(\lambda_{i-1},\lambda_{i})} & -(-1)^{\lambda_i}2 & 0 & Q_{(\lambda_{i},\lambda_{i+2})} & \cdots & Q_{(\lambda_{i},\lambda_n)} \\\hline
                \Block{3-3}<\fontsize{20}{50}>{M_3} &&& 0 & -Q_{(\lambda_{i},\lambda_{i+2})} & \Block{3-3}<\fontsize{20}{50}>{M_4} && \\
                 &&& \vdots & \vdots &&& \\
                 &&& 0 & -Q_{(\lambda_{i},\lambda_n)} &&&
            \end{pNiceArray}.}\]
        Therefore, if we expand along the $i$th row/column, we get
        \begin{align*}
            Q_{B_i\lambda}&=(-1)^{i-1}(-1)^{i+1}(-1)^{\lambda_i}2\Pf(M(\lambda\setminus\{\lambda_i,\lambda_{i+1}\}))\\
            &=(-1)^{\lambda_i}2Q_{\lambda\setminus\{\lambda_i,\lambda_{i+1}\}}.
        \end{align*}
        Hence, we have proven the third case, and the fourth case follows from this argument since $B_i^2\lambda=\lambda$.
    \end{proof}

    Therefore, we see that for any composition $\lambda\in\Z^n$, we may reorder the parts to be weakly decreasing except possibly with negative parts first, and doing so will only change $Q_\lambda$ by a nonzero coefficient. 

    \subsection{Removing Negative Parts}

    Now, note that if a composition $\lambda$ has a repeated nonzero part, then $M(\lambda)$ has two identical rows/columns, and hence $Q_\lambda=0$. So, we will consider functions of the form $Q_{p\lambda}$, where $\lambda\in\Z^n$ is a strict partition and $p\in\Z$ is a (possibly negative) integer. In order to write a formula for $Q_{(-p)\lambda}$ ($p>0$), we first define, for any strict composition $\lambda\in\Z^n$ and any integer $p\in\Z$,
    \[Q_{\lambda\setminus\{p\}}:=
            \begin{cases}
                Q_{\lambda\setminus\{\lambda_i\}}&\text{if }p=\lambda_i,\\
                0&\text{otherwise}.
              \end{cases}\]
    In particular, we have $Q_{(p)\setminus\{p\}}=Q_0=1$. It is convenient to think of $\lambda\setminus\{p\}$ as a set difference \emph{only} when $p=\lambda_i$ for some $i$. Note that the notation $Q_{\lambda\setminus\{\lambda_i\}}$ specifies the \emph{index} of the part being remove (and so in that case $\lambda$ need not be strict), whereas $Q_{\lambda\setminus\{p\}}$ specifies the \emph{value} of the part being removed (hence the requirement that $\lambda$ is strict).
    
    \begin{proposition}\label{Q_(-p)lambda}
        Let $p\in\Z$, $p>0$, be a positive integer and let $\lambda\in\Z^n$ be a strict partition, then
        \[Q_{(-p)\lambda}=(-1)^{p+\ind(\lambda,p)+1}2Q_{\lambda\setminus\{p\}}.\]
    \end{proposition}
    
    \begin{proof}
        We may assume $n$ is odd, so then we have
        \[Q_{(-p)\lambda}=\Pf
            \begin{pNiceArray}{c|ccc}
                0 & Q_{(p,\lambda_1)} & \cdots & Q_{(p,\lambda_n)} \\\hline
                -Q_{(p,\lambda_1)} & \Block{3-3}<\fontsize{30}{50}>{M(\lambda)}\\
                \vdots\\
                -Q_{(p,\lambda_n)}
            \end{pNiceArray}.\]
            If $p\neq \lambda_j$ for all $j$, then by (\ref{Q_(r,-r)}) we see that the first row/column is all zeros. Therefore the Pfaffian is $0$, and so the identity holds since by definition we have $Q_{\lambda\setminus\{p\}}=0$. Otherwise, suppose $p=\lambda_j$ for some $j$. Then by (\ref{Q_(r,-r)}), we have
        \[Q_{(-p)\lambda}=\Pf\begin{pNiceArray}{c|ccccccc}
            0 & 0& \cdots & 0 & (-1)^p2 & 0 & \cdots & 0\phantom{..} \\\hline
            0 & \Block{7-7}<\fontsize{30}{50}>{M(\lambda)}\\
            \vdots\\
            0\\
            -(-1)^p2\\
            0\\
            \vdots\\
            0
        \end{pNiceArray},\]
        where the nonzero entry of the first row appears in the $(j+1)$th column. So, if we expand along the first row/column, we get
        \begin{align*}
            Q_{(-p)\lambda}&=(-1)^{j+1}(-1)^p2\Pf(M(\lambda\setminus\{\lambda_j\})\\
            &=(-1)^{p+j+1}2Q_{\lambda\setminus\{\lambda_j\}}.
        \end{align*}
    \end{proof}

    If $\lambda$ is a strict composition with multiple negative parts, then we may repeat this argument to get a formula for $Q_\lambda$. Consequently, we have shown that the set of functions $Q_\lambda$ for strict partitions $\lambda\in\mathcal{SP}$ is a spanning set of $\Gamma$.

    \section{Vertex Operator Identity}
    
    The purpose of this section is to use Pfaffian formulation of Schur's Q-functions to prove a particular vertex operator identity. First, we must introduce some additional definitions and notation.

    \subsection{Inner Products and Adjoints}

    We define an inner product $(\cdot,\cdot)$ on $\Gamma$ by requiring that the $Q_\lambda$ form an orthogonal basis. Thus, we set
    \[(Q_\lambda,Q_\mu)=2^{\ell(\lambda)}\delta_{\lambda\mu}\]
    for strict partitions $\lambda$ and $\mu$. Let $z$ be an indeterminate; in this section we will work over the ring $\C[z]\ox\Gamma$ of Schur's $Q$-functions with coefficients in $\C[z]$. We extend the inner product to $\C[z]\ox\Gamma$ by $\C[z]$-linearity.
    
    For any $F\in\Gamma$, we let $F^\perp$ denote the adjoint of multiplication by $F$ with respect to $(\cdot,\cdot)$,
    \[(FQ_\lambda,Q_\mu)=(Q_\lambda,F^\perp Q_\mu).\]
    If $F=\sum_n F_nz^n$ is an infinite series, we denote $F^\perp:=\sum_n z^nF_n^\perp$.
    
    \begin{proposition}
        For all partitions $\lambda$ and $\mu$, we have
        \[Q_\mu^\perp Q_\lambda=2^{\ell(\mu)}Q_{\lambda/\mu}.\]
    \end{proposition}
    
    \begin{proof}
        Recall \cite{jozefiak-pragacz:1991} that for any $F\in\Gamma$ we have
        \[\left(Q_{\lambda/\mu},F \right)=\left(Q_\lambda,2^{-\ell(\mu)}Q_\mu F\right),\]
        which is a fundamental identity of the skew functions $Q_{\lambda/\mu}$. Using linearity and the definition of adjoint on the RHS, we get
        \[\left(Q_\lambda,2^{-\ell(\mu)}Q_\mu F\right)=\left(2^{-\ell(\mu)}Q_\mu^\perp Q_\lambda,F\right).\]
        Together, we have
        \[Q_{\lambda/\mu}=2^{-\ell(\mu)}Q_\mu^\perp Q_\lambda.\]
    \end{proof}

    In particular, we will make use of the following case where $\ell(\mu)\leq1$.
    
    \begin{corollary}\label{adjoint}
        For any partition $\lambda$, we have
        \begin{align*}
            q_r^\perp Q_\lambda&=2Q_{\lambda/(r)}\qquad(r\geq1),\\
            q_0^\perp Q_\lambda&=Q_{\lambda}.
        \end{align*}
    \end{corollary}

    \subsection{The Main Identity}

    The following identity has been proven by Jing \cite{jing:1991a} for Schur Q-functions $Q_\lambda(x_1,x_2,\ldots;t)$ and in general for Hall-Littlewood functions \cite{jing:1991b} using the language of vertex operators, noting that Hall-Littlewood functions specialize to Schur's $Q$-functions when $t=-1$ and to Schur functions  when $t=0$. The analoguous identity for Schur functions was proven by Carre and Thibon \cite{carre-thibon:1992} using a determinantal approach with the Jacobi-Trudi identity.
    
    \begin{theorem}\label{vertex operator identity}
        Let $\lambda$ be a partition, then we have
        \[\sum_{p\in\Z}Q_{p\lambda}z^p=\kappa_z\cdot \kappa_{-1/z}^\perp Q_\lambda.\]
    \end{theorem}
    
    Our first step is to expand $\kappa_z\cdot \kappa_{-1/z}^\perp Q_\lambda$ as a sum.
    
    \begin{lemma}\label{expanded RHS}
        Let $\lambda$ be a partition, then we have
        \[\kappa_z\cdot \kappa_{-1/z}^\perp Q_\lambda=\sum_{p\in\Z}z^p\left(q_pQ_\lambda+2\sum_{r\geq1}(-1)^rq_{p+r}Q_{\lambda/(r)}\right).\]
    \end{lemma}
    
    \begin{proof}
        First, we multiply $\kappa_z$ and $\kappa_{-1/z}^\perp$ to get
        \begin{align*}
            \kappa_z\cdot \kappa_{-1/z}^\perp Q_\lambda&=\sum_{i\geq0}q_iz^i\cdot\sum_{j\geq0} (-z)^{-j}q_j^\perp Q_\lambda\\
            &=\sum_{i\geq0}\sum_{j\geq0}q_iz^i(-z)^{-j}q_j^\perp Q_\lambda\\
            &=\sum_{p\in\Z}\sum_{r\geq0}q_{p+r}z^{p+r}(-z)^{-r}q_r^\perp Q_\lambda
        \end{align*}
        since $q_{p+r}=0$ for $p+r<0$. Then, we rewrite this to get
        \[\sum_{p\in\Z}z^p\sum_{r\geq0}(-1)^rq_{p+r}q_r^\perp Q_\lambda.\]
        Finally, we apply Corollary \ref{adjoint} to $q_r^\perp Q_\lambda$ in the sum to get
        \begin{align*}
            &\sum_{p\in\Z}z^p\left(q_{p+0}q_0^\perp Q_\lambda+\sum_{r\geq1}(-1)^rq_{p+r}q_r^\perp Q_\lambda\right)\\
            &\quad=\sum_{p\in\Z}z^p\left(q_pQ_\lambda+2\sum_{r\geq1}(-1)^rq_{p+r}Q_{\lambda/(r)}\right)
        \end{align*}
    \end{proof}
    
    So, to prove Theorem  \ref{vertex operator identity}, it suffices to show that
    \[Q_{p\lambda}=q_pQ_\lambda+2\sum_{r\geq1}(-1)^rq_{p+r}Q_{\lambda/(r)}.\]
    
    We need several more identities to finish the proof. First, we will need to write both $Q_{p\lambda}$ and $Q_{\lambda/(r)}$ as a sums of Schur's $Q$-functions of the form $Q_{\lambda\setminus\{\lambda_i\}}$.
    
    \begin{lemma}\label{Q_plam}
        Let $\lambda\in\Z^n$ be a partition, then for all integers $p\in\Z$ we have
        \[Q_{p\lambda}=
            \begin{cases}
                \phantom{q_pQ_\lambda}-\dsp\sum_{i=1}^n(-1)^iQ_{(p,\lambda_i)}Q_{\lambda\setminus\{\lambda_i\}} & \text{if }n\text{ is odd},\\
                q_pQ_\lambda-\dsp\sum_{i=1}^n(-1)^iQ_{(p,\lambda_i)}Q_{\lambda\setminus\{\lambda_i\}} & \text{if }n\text{ is even}.
            \end{cases}
        \]
    \end{lemma}
    
    \begin{proof}
        If $n$ is odd, then $n+1$ is even, so we have
        \begin{align*}
            Q_{p\lambda}&=\Pf\begin{pNiceArray}{c|ccc}
                0 & Q_{(p,\lambda_1)} & \cdots & Q_{(p,\lambda_n)} \\\hline
                -Q_{(p,\lambda_1)} & \Block{3-3}<\fontsize{30}{50}>{M(\lambda)}\\
                \vdots\\
                -Q_{(p,\lambda_n)}
            \end{pNiceArray}\\
            &=\sum_{j=1}^n(-1)^{j+1}Q_{(p,\lambda_j)}\Pf(M(\lambda\setminus\{\lambda_j\}))
        \end{align*}
        by taking the Laplace expansion along the first row/column. Then, note that $\Pf(M(\lambda\setminus\{\lambda_j\}))=Q_{\lambda\setminus\{\lambda_j\}}$.
        
        If $n=$ is even, then we have
        \begin{align*}
            Q_{p\lambda}&=
                \Pf\begin{pNiceArray}{c|cccc}
                    0 & Q_{(p,\lambda_1)} & \cdots & Q_{(p,\lambda_n)} & Q_{(p,0)} \\\hline
                    -Q_{(p,\lambda_1)} & \Block{4-4}<\fontsize{30}{50}>{M(\lambda0)}\\
                    \vdots\\
                    -Q_{(p,\lambda_n)}\\
                    -Q_{(p,0)}
                \end{pNiceArray}\\
                &=\sum_{j=1}^n(-1)^{j+1}Q_{(p,\lambda_j)}\Pf(M(\lambda0\setminus\{\lambda_j\}))+(-1)^{n+2}Q_{(p,0)}\Pf(M(\lambda0))
        \end{align*}
        by taking the Laplace expansion along the first row/column. Then, note that $\Pf(M(\lambda0\setminus\{\lambda_j\}))=Q_{\lambda\setminus\{\lambda_j\}}$. Additionally, we have that $Q_{(p,0)}=q_p$, and $\Pf M(\lambda0)=Q_\lambda$.
    \end{proof}
    
    The following identity is a special case of an identity by J{\'o}zefiak and Pragacz \cite{jozefiak-pragacz:1991}.
    
    \begin{lemma}[J{\'o}zefiak-Pragacz]\label{Q_lam/(r)}
        Let $\lambda\in\Z^n$ be a partition, then for all positive integers $r\in\Z^+$ we have
        \[Q_{\lambda/(r)}=-\sum_{i=1}^n(-1)^iq_{\lambda_i-r}Q_{\lambda\setminus\{\lambda_i\}}.\]
    \end{lemma}
    
    \begin{proof}
        If $n$ is odd, then we have
        \begin{align*}
            Q_{\lambda/(r)}&=\Pf
                \begin{pNiceArray}{ccc|c}
                    \Block{3-3}<\fontsize{30}{50}>{M(\lambda)} &&& q_{\lambda_1-r}\\
                    &&& \vdots\\
                    &&& q_{\lambda_n-r}\\\hline
                    -q_{\lambda_1-r} & \cdots & -q_{\lambda_n-r} & 0
                \end{pNiceArray}\\
            &=(-1)^{n+1-1}\sum_{j=1}^n(-1)^jq_{\lambda_j-r}\Pf(M(\lambda\setminus\{\lambda_j\}))
        \end{align*}
        by taking the Laplace expansion along the last row/column. Then, note that $\Pf(M(\lambda\setminus\{\lambda_j\}))=Q_{\lambda\setminus\{\lambda_j\}}$.
        
        If $n$ is even, then we have
        \begin{align*}
            Q_{\lambda/(r)}&=\Pf
                \begin{pNiceArray}{cccc|c}
                    \Block{4-4}<\fontsize{30}{50}>{M(\lambda0)} &&&& q_{\lambda_1-r}\\
                    &&&& \vdots\\
                    &&&& q_{\lambda_n-r}\\
                    &&&& q_{0-r}\\\hline
                    -q_{\lambda_1-r} & \cdots & -q_{\lambda_n-r} & -q_{0-r} & 0
                \end{pNiceArray}\\
            &=(-1)^{n+2-1}\sum_{j=1}^n(-1)^jq_{\lambda_j-r}\Pf(M(\lambda0\setminus\{\lambda_j\}))
        \end{align*}
        by taking the Laplace expansion along the last row/column, since $q_{-r}=0$. Then, note that $\Pf(M(\lambda0\setminus\{\lambda_j\}))=Q_{\lambda\setminus\{\lambda_j\}}$.
    \end{proof}
    
    Note that our equation in Lemma \ref{Q_plam} for $Q_{p\lambda}$ is missing the term $q_pQ_\lambda$ when $\ell(\lambda)$ is odd. However, this is accounted for due to Proposition \ref{alternating sum identity}. So, we are able to prove the following identity, which completes the proof of Theorem \ref{vertex operator identity}.
    
    \begin{theorem}\label{Q_plam=q_pQ_lam+2...}
        Let $\lambda\in\Z^n$ be a partition, then for all $p\in\Z$ we have
        \[Q_{p\lambda}=q_pQ_\lambda+2\sum_{r\geq1}(-1)^rq_{p+r}Q_{\lambda/(r)}.\]
    \end{theorem}
    
    \begin{proof}
        First, suppose $n$ is odd. Then, we substitute the formula (\ref{Q_(r,s)}) for $Q_{(r,s)}$ into Lemma \ref{Q_plam} to get
        \begin{align*}
            Q_{p\lambda}&=-\sum_{i=1}^{\ell(\lambda)}(-1)^iQ_{(p,\lambda_i)}Q_{\lambda\setminus\{\lambda_i\}}\\
            &=-\sum_{i=1}^{\ell(\lambda)}(-1)^i\left(q_pq_{\lambda_i}+2\sum_{r\geq1}(-1)^rq_{p+r}q_{\lambda_i-r}\right)Q_{\lambda\setminus\{\lambda_i\}}.
        \end{align*}
        We expand out this sum to get that this is equal to 
        \begin{align*}
            -q_p\sum_{i=1}^{\ell(\lambda)}(-1)^iq_{\lambda_i}Q_{\lambda\setminus\{\lambda_i\}}-2\sum_{i=1}^{\ell(\lambda)}(-1)^i\sum_{r\geq1}(-1)^rq_{p+r}q_{\lambda_i-r}Q_{\lambda\setminus\{\lambda_i\}}.
        \end{align*}
        By Lemma \ref{alternating sum identity}, we see that the sum in the first term is $-Q_\lambda$, so we have
        \begin{align*}
            q_pQ_\lambda-2\sum_{i=1}^{\ell(\lambda)}(-1)^i\sum_{r\geq1}(-1)^rq_{p+r}q_{\lambda_i-r}Q_{\lambda\setminus\{\lambda_i\}}.
        \end{align*}
        Then, we rewrite the second two sums to get
        \begin{align*}
            q_pQ_\lambda+2\sum_{r\geq1}(-1)^rq_{p+r}\left(-\sum_{i=1}^{\ell(\lambda)}(-1)^iq_{\lambda_i-r}Q_{\lambda\setminus\{\lambda_i\}}\right).
        \end{align*}
        Finally, we can apply Lemma \ref{Q_lam/(r)} to the sum inside the parentheses, and so we have
        \begin{align*}
            Q_{p\lambda}=q_pQ_\lambda+2\sum_{r\geq1}(-1)^rq_{p+r}Q_{\lambda/(r)}.
        \end{align*}
        
        Next, suppose $n$ is even. We repeat the first steps of the odd case to get
        \[Q_{p\lambda}=q_pQ_\lambda-q_p\sum_{i=1}^{\ell(\lambda)}(-1)^iq_{\lambda_i}Q_{\lambda\setminus\{\lambda_i\}}-2\sum_{i=1}^{\ell(\lambda)}(-1)^i\sum_{r\geq1}(-1)^rq_{p+r}q_{\lambda_i-r}Q_{\lambda\setminus\{\lambda_i\}},\]
        where there is now an extra term $q_pQ_\lambda$. However, by Lemma \ref{alternating sum identity} we have that the first sum is $0$, and so we get
        \[q_pQ_\lambda-2\sum_{i=1}^{\ell(\lambda)}(-1)^i\sum_{r\geq1}(-1)^rq_{p+r}q_{\lambda_i-r}Q_{\lambda\setminus\{\lambda_i\}},\]
        which is the same as in the odd case. So, we repeat the final steps as the odd case to finish the proof.
    \end{proof}
    
    \begin{remark}
        We could have used Proposition \ref{Q_lambda0} to simplify several statements and proofs in this section without having to break into even and odd cases. However, we would have less information about the differences due to the parity of $\ell(\lambda)$.
    \end{remark}
    
    \section{Expanding $Q_\lambda$ Using Skew Functions}

    In this section, we will expand $Q_\lambda$ as a sum involving skew Schur's $Q$-functions of the form $Q_{\lambda/(i)}$ in two different ways. In the first method, $i$ will range over all nonnegative integers $r\geq0$. In the second method, $i$ will range over just the parts $\lambda_1,\lambda_2,\ldots$ of $\lambda$.
    
    \subsection{Sum of $Q_{\lambda/(r)}$}

    Many of our identities have involved the function $Q_{p\lambda}$. However, we may easily find identities for $Q_\lambda$ by setting $p=0$. First, we specialize the final result from last section.
    
    \begin{corollary}\label{alternating sum of skew}
        Let $\lambda$ be a partition, then
        \[Q_\lambda=\begin{cases}
            \dsp-\sum_{r\geq1}(-1)^rq_rQ_{\lambda/(r)}&\text{if }\ell(\lambda)\text{ is odd},\\
            \dsp\Neg\sum_{r\geq0}(-1)^rq_rQ_{\lambda/(r)}&\text{if }\ell(\lambda)\text{ is even}.
        \end{cases}\]
    \end{corollary}
    
    \begin{proof}
        Let $p=0$ in Theorem \ref{Q_plam=q_pQ_lam+2...}, then we have
        \[Q_{0\lambda}=Q_\lambda+2\sum_{r\geq1}(-1)^rq_rQ_{\lambda/(r)}.\]
        Repeatedly applying Proposition \ref{B_iQ_lambda} to $Q_{0\lambda}$, we see that we get
        \[(-1)^{\ell(\lambda)}Q_{\lambda0}=Q_\lambda+2\sum_{r\geq1}(-1)^rq_rQ_{\lambda/(r)}.\]
        Then, since $Q_{\lambda0}=Q_\lambda$ we have
        \[-\sum_{r\geq1}(-1)^rq_rQ_{\lambda/(r)}=
            \begin{cases}
                Q_\lambda&\text{if }\ell(\lambda)\text{ is odd},\\
                0&\text{if }\ell(\lambda)\text{ is even}.
            \end{cases}\]
        Subtracting $q_0Q_{\lambda/(0)}=Q_\lambda$ from these equations, we get
        \[\sum_{r\geq0}(-1)^rq_rQ_{\lambda/(r)}=
            \begin{cases}
                0&\text{if }\ell(\lambda)\text{ is odd},\\
                Q_\lambda&\text{if }\ell(\lambda)\text{ is even}.
            \end{cases}\]
    \end{proof}
    
    \subsection{Sum of $Q_{\lambda/(\lambda_k)}$}

    Now, we will write a sum ranging over the parts of $\lambda$. In order to find this sum, we will need to write a similar decomposition of Schur's $Q$-functions of the form $Q_{\lambda\setminus\{\lambda_k\}}$. Then, we may substitute the decomposition into Lemma \ref{Q_plam} to get our decomposition.
    
    
    
    \begin{lemma}\label{Q_lam/lam_k in terms of skew Schur}
        Let $\lambda$ be a partition and fix an integer $k$ such that $1\leq k\leq\ell(\lambda)$, then
        \[Q_{\lambda\setminus\{\lambda_k\}}=-\sum_{i=1}^k(-1)^ir_{i,k}^\lambda Q_{\lambda/(\lambda_i)}\]
        where $r_{i,k}^\lambda$ is a polynomial in the variables $q_1,q_2,\ldots,q_{\lambda_i-\lambda_k}$. Specifically, we have $r_{i,i}^\lambda=1$, and recursively define (for $k>i$)
        \[r_{i,k}^\lambda=-\sum_{j=i}^{k-1}(-1)^{j-k}q_{\lambda_j-\lambda_k}r_{i,j}^\lambda.\]
    \end{lemma}
    
    \begin{proof}
        Recall from Lemma \ref{Q_lam/(r)} that we have
        \[Q_{\lambda/(r)}=-\sum_{i=1}^n(-1)^iq_{\lambda_i-r}Q_{\lambda\setminus\{\lambda_i\}}.\]
        First, note that if $r=\lambda_k$, then the sum only needs to go from $1$ to $k$ since $q_{\lambda_i-r}=0$ for $\lambda_i<r=\lambda_k$. Additionally, we have that $q_{\lambda_k-\lambda_k}=q_0=1$, so we have
        \[-Q_{\lambda/(\lambda_k)}=(-1)^kQ_{\lambda\setminus\{\lambda_k\}}+\sum_{i=1}^{k-1}(-1)^iq_{\lambda_i-\lambda_k}Q_{\lambda\setminus\{\lambda_i\}}.\]
        Then, we can solve for $Q_{\lambda\setminus\{\lambda_k\}}$, and we get
        \begin{equation}\label{Q_lam/lam_k}
            Q_{\lambda\setminus\{\lambda_k\}}=(-1)^{1-k}Q_{\lambda/(\lambda_k)}-\sum_{i=1}^{k-1}(-1)^{i-k}q_{\lambda_i-\lambda_k}Q_{\lambda\setminus\{\lambda_i\}}.
        \end{equation}
        
        Now, we can use this identity to prove the Lemma. We proceed by induction on $k$. First, if $k=1$, then from (\ref{Q_lam/lam_k}) we see that
        \begin{align*}
            Q_{\lambda\setminus\{\lambda_1\}}&=(-1)^{1-1}Q_{\lambda/(\lambda_1)}\\
            &=-\left((-1)^1r_{1,1}^\lambda Q_{\lambda/(\lambda_1)}\right)
        \end{align*}
        since $r_{1,1}^\lambda=1$. Suppose the Lemma is true for $Q_{\lambda\setminus\{\lambda_k\}}$, then by (\ref{Q_lam/lam_k}) we have
        \begin{align*}
            Q_{\lambda\setminus\{\lambda_{k+1}\}}&=(-1)^{1-(k+1)}Q_{\lambda/(\lambda_{k+1})}-\sum_{i=1}^{k}(-1)^{i-(k+1)}q_{\lambda_i-\lambda_{k+1}}Q_{\lambda\setminus\{\lambda_i\}}\\
            &=(-1)^{k}Q_{\lambda/(\lambda_{k+1})}-\sum_{i=1}^{k}(-1)^{i-(k+1)}q_{\lambda_i-\lambda_{k+1}}\left(-\sum_{j=1}^i(-1)^jr_{j,i}^\lambda Q_{\lambda/(\lambda_j)}\right)
        \end{align*}
        since $i\leq k$. We can rewrite this sum to get
        \[(-1)^{k}Q_{\lambda/(\lambda_{k+1})}+\sum_{j=1}^k(-1)^j\left(\sum_{i=j}^k(-1)^{i-(k+1)}q_{\lambda_i-\lambda_{k+1}}r_{j,i}^\lambda \right)Q_{\lambda/(\lambda_j)},\]
        and so by the definition of $r_{i,k}^\lambda$ we have
        \begin{align*}
            &-(-1)^{k+1}r_{k+1,k+1}^\lambda Q_{\lambda/(\lambda_{k+1})}-\sum_{j=1}^k(-1)^jr_{j,k+1}^\lambda Q_{\lambda/(\lambda_j)}\\
            &\qquad=-\sum_{j=1}^{k+1}(-1)^jr_{j,k+1}^\lambda Q_{\lambda/(\lambda_j)}.
        \end{align*}
    \end{proof}

    Now, we can use this to get a decomposition of $Q_{p\lambda}$.
    
    \begin{proposition}\label{Q_plam in terms of skew Schur}
        Let $\lambda$ be a partition and $p\in\Z$, then
        \[Q_{p\lambda}=
            \begin{cases}
             \phantom{ q_pQ_\lambda+}\,\dsp\sum_{k=1}^{\ell(\lambda)}\left(\sum_{i=k}^{\ell(\lambda)}(-1)^{i+k}Q_{(p,\lambda_i)}r_{k,i}^\lambda\right)Q_{\lambda/(\lambda_k)}&\text{if }\ell(\lambda)\text{ is odd},\\
             q_pQ_\lambda+\dsp\sum_{k=1}^{\ell(\lambda)}\left(\sum_{i=k}^{\ell(\lambda)}(-1)^{i+k}Q_{(p,\lambda_i)}r_{k,i}^\lambda\right)Q_{\lambda/(\lambda_k)}&\text{if }\ell(\lambda)\text{ is even}.
            \end{cases}		
        \]
    \end{proposition}
    
    \begin{proof}
        Recall from Lemma \ref{Q_plam} that we have
        \[Q_{p\lambda}=
            \begin{cases}
                \phantom{q_pQ_\lambda}-\dsp\sum_{i=1}^n(-1)^iQ_{(p,\lambda_i)}Q_{\lambda\setminus\{\lambda_i\}} & \text{if }n\text{ is odd},\\
                q_pQ_\lambda-\dsp\sum_{i=1}^n(-1)^iQ_{(p,\lambda_i)}Q_{\lambda\setminus\{\lambda_i\}} & \text{if }n\text{ is even}.
            \end{cases}\]
        From Lemma \ref{Q_lam/lam_k in terms of skew Schur} we can replace $Q_{\lambda\setminus\{\lambda_i\}}$ so that the sum becomes
        \[-\sum_{i=1}^{\ell(\lambda)}(-1)^{i}Q_{(p,\lambda_i)}Q_{\lambda\setminus\{\lambda_i\}}=\sum_{i=1}^{\ell(\lambda)}(-1)^{i}Q_{(p,\lambda_i)}\sum_{j=1}^i(-1)^jr_{j,i}^\lambda Q_{\lambda/(\lambda_j)}.\]
        Finally, we can swap the sums to get
        \[=\sum_{j=1}^{\ell(\lambda)}\left(\sum_{i=j}^{\ell(\lambda)}(-1)^{i+j}Q_{(p,\lambda_i)}r_{j,i}^\lambda\right)Q_{\lambda/(\lambda_j)}.\]
    \end{proof}
    
    Like before, we are able to set $p=0$ to get an analogous statement about $Q_\lambda$ as a sum of skew functions.
    
    \begin{corollary}\label{alternating sum of skew2}
        Let $\lambda$ be a partition, then
        \[Q_{\lambda}=
            \begin{cases}
             \dsp-\sum_{k=1}^{\ell(\lambda)}(-1)^k a_k^\lambda Q_{\lambda/(\lambda_k)}&\text{if }\ell(\lambda)\text{ is odd},\\
             \dsp\Neg\sum_{k=0}^{\ell(\lambda)}(-1)^k a_k^\lambda Q_{\lambda/(\lambda_k)}&\text{if }\ell(\lambda)\text{ is even},
            \end{cases}\]
        where we set $\lambda_0:=0$, and we define the coefficients
        \[a_k^\lambda:=
            \begin{cases}
                \dsp-\sum_{i=k}^{\ell(\lambda)}(-1)^iq_{\lambda_i}r_{k,i}^\lambda&1\leq k\leq\ell(\lambda),\\
                1&k=0,
            \end{cases}\]
        where $r_{k,i}^\lambda$ are defined as in Lemma \ref{Q_lam/lam_k in terms of skew Schur}.
    \end{corollary}
    
    \begin{proof}
        Let $p=0$ in Lemma \ref{Q_plam in terms of skew Schur}, then we have
        \[Q_{0\lambda}=
            \begin{cases}
             \phantom{q_0Q_\lambda+}\,\dsp\sum_{k=1}^{\ell(\lambda)}\left(\sum_{i=k}^{\ell(\lambda)}(-1)^{i+k}(-q_{\lambda_i})r_{k,i}^\lambda\right)Q_{\lambda/(\lambda_k)}&\text{if }\ell(\lambda)\text{ is odd},\\
             q_0Q_\lambda+\dsp\sum_{k=1}^{\ell(\lambda)}\left(\sum_{i=k}^{\ell(\lambda)}(-1)^{i+k}(-q_{\lambda_i})r_{k,i}^\lambda\right)Q_{\lambda/(\lambda_k)}&\text{if }\ell(\lambda)\text{ is even}.
            \end{cases}		
        \]
        Then, repeatedly apply Proposition \ref{B_iQ_lambda} to get $Q_{0\lambda}=(-1)^{\ell(\lambda)}Q_\lambda$, and so we have
        \[\sum_{k=1}^{\ell(\lambda)}\left(\sum_{i=k}^{\ell(\lambda)}(-1)^{i+k}q_{\lambda_i}r_{k,i}^\lambda\right)Q_{\lambda/(\lambda_k)}=
            \begin{cases}
                Q_\lambda &\text{if }\ell(\lambda)\text{ is odd},\\
                0 &\text{if }\ell(\lambda)\text{ is even}.
            \end{cases}
        \]
        Equivalently, this is
        \[-\sum_{k=1}^{\ell(\lambda)}(-1)^ka_k^\lambda Q_{\lambda/(\lambda_k)}=
            \begin{cases}
                Q_\lambda &\text{if }\ell(\lambda)\text{ is odd},\\
                0 &\text{if }\ell(\lambda)\text{ is even}.
            \end{cases}
        \]
        Subtracting $a_0^\lambda Q_{\lambda/(0)}=Q_\lambda$, we have
        \[\sum_{k=0}^{\ell(\lambda)}(-1)^ka_k^\lambda Q_{\lambda/(\lambda_k)}=
            \begin{cases}
                0 &\text{if }\ell(\lambda)\text{ is odd},\\
                Q_\lambda &\text{if }\ell(\lambda)\text{ is even}.
            \end{cases}
        \]
    \end{proof}
    
    \subsection{Specializing $r_{i,k}^\lambda$ and $a_k^\lambda$ to Staircase Partitions}

    As we have seen, it is useful to work with sums of Schur's $Q$-functions of the form $Q_{\lambda\setminus\{\lambda_k\}}$. So, it may be useful to study the decompositions of the functions $Q_{\lambda\setminus\{\lambda_k\}}$, and in particular the coefficients $r_{i,k}^\lambda$ and $a_k^\lambda$ from Lemma \ref{Q_lam/lam_k in terms of skew Schur} and Corollary \ref{alternating sum of skew2}. First, we find simple formulas for the coefficients in the case where $\lambda$ is a staircase partition
    \[\delta(n):=(n-1,n-2,\ldots,2,1)\in\Z^{n-1}\]
    for $n\geq1$. To start, we see that $r_{i,k}^{\delta(n)}$ depends only on the difference $k-i$.
    
    
    \begin{lemma}
        Fix $i\leq k<n$, then for all $j$ such that $-i<j<n-k$, we have
        \[r_{i,k}^{\delta(n)}=r_{i+j,k+j}^{\delta(n)}.\]
        In particular,
        \[r_{i,k}^{\delta(n)}=r_{1,k-i+1}^{\delta(n)}.\]
    \end{lemma}
    
    \begin{proof}
        This is clear from the definitions of $r_{i,k}^\lambda$ and $\delta(n)$ since $r_{i,k}^\lambda$ depends on the differences between the parts $\lambda_i,\ldots,\lambda_k$.
    \end{proof}

    So, it suffices to find a formula for $r_{1,k}^{\delta(n)}$.
    
    \begin{lemma}
        For all $k<n$, we have
        \[r_{1,k}^{\delta(n)}=q_{k-1}.\]
    \end{lemma}
    
    \begin{proof}
        We proceed by induction on $k$. First, we see that
        \[r_{1,1}^{\delta(n)}=1=q_0=q_{1-1}.\]
        Next, we have
        \[r_{1,k}^{\delta(n)}=-\sum_{j=1}^{k-1}(-1)^{j-k}q_{k-j}r_{1,j}^{\delta(n)}.\]
        By the induction hypothesis, we have that this is
        \[-\sum_{j=1}^{k-1}(-1)^{j-k}q_{k-j}q_{j-1}.\]
        We can rewrite this sum as
        \[(-1)^k\left(\left(\sum_{r+s=k-1}(-1)^rq_sq_r\right)-(-1)^{k-1}q_0q_{k-1}\right).\]
        It is well-known \cite[p.~251]{macdonald:1995} that this sum is $0$ for $k\geq2$, hence we are left with $q_{k-1}$.
    \end{proof}
    
    \begin{proposition}\label{r_ik^del}
    For all $i\leq k<n$, we have $r_{i,k}^{\delta(n)}=q_{k-i}$.
    \end{proposition}
    
    \begin{proof}
        This is an immediate consequence of the previous two Lemmas.
    \end{proof}
    

    Now, we can use our formula for $r_{i,k}^{\delta(n)}$ to get a formula for $a_k^{\delta(n)}$.
    
    \begin{proposition}
        \[a_k^{\delta(n)}:=
            \begin{cases}
                (-1)^nq_{n-k}&1\leq k\leq n-1,\\
                1&k=0.
            \end{cases}\]
    \end{proposition}
    
    \begin{proof}
        Using Proposition \ref{r_ik^del}, for $k\geq1$ we have
        \begin{align*}
            a_k^{\delta(n)}&=-\sum_{i=k}^{n-1}(-1)^iq_{\delta_i}r_{k,i}^{\delta(n)}\\
            &=-\sum_{i=k}^{n-1}(-1)^iq_{n-i}q_{i-k}
        \end{align*}
        Next, we write the sum as
        \[(-1)^{k+1}\sum_{j=0}^{n-k-1}(-1)^jq_jq_{n-k-j},\]
        which is equal to
        \[(-1)^{k+1}\left(\left(\sum_{r+s=n-k}(-1)^rq_rq_s\right)-(-1)^{n-k}q_{n-k}q_0\right).\]
        Once again the sum is $0$, and so we get $(-1)^nq_{n-k}$.
    \end{proof}

    \subsection{The Coefficients $r_{i,k}^\lambda$ and the $A_{n-1}$ Root System}

    The $A_{n-1}$ root system has \emph{positive roots} $\Phi^+:=\{\varepsilon_i-\varepsilon_j\,|\,1\leq i<j\leq n\}$, where $\varepsilon_1,\ldots,\varepsilon_n$ are linear functionals of the Cartan subalgebra, and here we view them as formal symbols. For any $i\leq k$, we say that a \emph{decomposition} of $\beta\in\Phi^+$ is a subset $D\subseteq\Phi^+$ of positive roots that sum to $\beta$. Additionally, we say that the only decomposition of $\varepsilon_i-\varepsilon_i=0$ is the empty set $\emptyset$. Otherwise, for $i<k$ it is easy to see that there are $2^{k-i-1}$ decompositions of $\varepsilon_i-\varepsilon_k$ (corresponding to subsets of $\{i,i+1,\ldots,k\}$ that contain both $i$ and $k$). For more details, see the Kostant partition function \cite{kostant:1959,hall:2015}.
    
    We may use these definitions to write the coefficients $r_{i,k}^\lambda$ as a closed sum for any partition $\lambda$.

    \begin{proposition}
        Let $\lambda\in\Z^n$ be a partition, and suppose $1\leq i\leq k\leq n$. Then we have
        \[r_{i,k}^\lambda=(-1)^{i-k}\sum_D (-1)^{\# D}q_D^\lambda,\]
        where the sum ranges over decompositions $D=\{\beta_1,\ldots,\beta_\ell\}\subseteq\Phi^+$ of $\varepsilon_i-\varepsilon_k$, $\# D$ is the cardinality of $D$, and we define
        \begin{align*}
            q_D^\lambda&:=q_{\beta_1}^\lambda\cdots q_{\beta_\ell}^\lambda\quad(\#D\geq1),\\
            q_\emptyset^\lambda&:=1,
        \end{align*}
        where for $\beta_j=\varepsilon_r-\varepsilon_s$ we have
        \[q_{\beta_j}^\lambda:=q_{\lambda_r-\lambda_s}.\]
    \end{proposition}


    \begin{proof}
       We proceed by induction on the difference $k-i\geq0$. First, when $k=i$, there is nothing to prove. Then, when $k=i+1$, we have from Lemma \ref{alternating sum of skew} that
        \[r_{i,i+1}^\lambda=-(-1)^{i-(i+1)}q_{\lambda_i-\lambda_{i+1}}\cdot 1.\]
        Since $\{\varepsilon_i-\varepsilon_{i+1}\}$ is the only subset of $\Phi^+$ such that the sum of its elements is $\lambda_i-\lambda_{i+1}$, we are done with the base case.

        Next, we see that we have 
        \begin{align*}
            r_{i,k}^\lambda&=-\sum_{j=i}^{k-1}(-1)^{j-k}q_{\lambda_j-\lambda_k}r_{i,j}^\lambda\\
            &=-\sum_{j=i+1}^{k-1}(-1)^{j-k}q_{\lambda_j-\lambda_k}
            \left((-1)^{i-j}\sum_D (-1)^{\# D}q_D^\lambda\right)-(-1)^{i-k}q_{\lambda_i-\lambda_k}
        \end{align*}
        by the induction hypothesis, where the sum is over decompositions of $\varepsilon_i-\varepsilon_j$. After rewriting, we get
        \[(-1)^{i-k}\left(\sum_{j=i+1}^{k-1}-q_{\lambda_j-\lambda_k}\sum_D (-1)^{\# D}q_D^\lambda-q_{\lambda_i-\lambda_k}\right).\]
        Notice that $-q_{\lambda_i-\lambda_k}=(-1)^{\#\{\varepsilon_i-\varepsilon_k\}}q_{\{\varepsilon_i-\varepsilon_k\}}^\lambda$. Also, for any decomposition $D$ of $\varepsilon_i-\varepsilon_j$, we have that the elements of $D\cup\{\lambda_j-\lambda_k\}$ sum to $\varepsilon_i-\varepsilon_k$. Therefore, we can combine everything into the desired sum, ranging over decompositions of $\varepsilon_i-\varepsilon_k$.
    \end{proof}
    
    \section{Further Identities}

    \subsection{A Skew Function Identity}

    The Schur functions $S_\lambda$ form a basis of the ring $\Lambda$ of symmetric functions. One identity that Schur functions satisfy is $S_{p\lambda/(p-k)}=S_{k\lambda}$ (see \cite[p.~396]{carre-thibon:1992}). We may prove a similar identity with Schur's $Q$-functions.
    
    \begin{proposition}\label{p>t+r}
        For all partitions $\lambda\in\Z^n$ and integers $k,p\in\Z$ such that $k\geq0$ and $p>\lambda_1+k$, we have
        \[Q_{p\lambda/(p-k)}=q_kQ_\lambda.\]
    \end{proposition}
    
    \begin{proof}
        First, note that we have $q_{\lambda_i-(p-k)}=0$ for all $i$ since we see that $\lambda_i-(p-k)\leq\lambda_1-(p-k)<0$. Next, we may assume that $n$ is even, so we have
        \begin{align*}
            Q_{p\lambda/(p-k)}&=\Pf\begin{pNiceArray}{cccc|c}
                \Block{4-4}<\fontsize{30}{50}>{M(p\lambda)} &&&& q_{p-(p-k)}\\
                &&&& q_{\lambda_1-(p-k)}\\
                &&&& \vdots\\
                &&&& q_{\lambda_n-(p-k)}\\\hline
                -q_{p-(p-k)} & -q_{\lambda_1-(p-k)} & \cdots & -q_{\lambda_n-(p-k)} & 0
            \end{pNiceArray}\\
            &=\Pf\begin{pNiceArray}{cccc|c}
                \Block{4-4}<\fontsize{25}{50}>{M(p\lambda)} &&&& q_k\\
                &&&& 0\\
                &&&& \vdots\\
                &&&& 0\\\hline
                -q_k & 0 & \cdots & 0 & 0
            \end{pNiceArray}\\
            &=(-1)^{n+2-1}(-1)^1q_k\Pf M(\lambda)
        \end{align*}
        by expanding along the last row/column.
    \end{proof}

    \subsection{An Alternating Identity}

    As we have seen, Schur's $Q$-functions satisfy several identities involving alternating sums. For example, for any fixed integer $n\in\Z$ we have the fundamental identity
    \[\sum_{r+s=n}(-1)^rq_rq_s=
        \begin{cases}
            0&n\neq0,\\
            1&n=0,
        \end{cases}\]
    (see \cite[p.~251]{macdonald:1995}). We now provide a similar identity involving Schur's $Q$-functions with two parts.

    
    \begin{proposition}
        Fix integers $p,n\in\Z$, then we have
        \[\sum_{r+s=n}(-1)^rq_rQ_{(p,s)}=
            \begin{cases}
                (-1)^n2q_{p+n}&n>0,\\
                q_p&n=0,\\
                0&n<0,
            \end{cases}\]
        and
        \[\sum_{r+s=n}(-1)^rq_rQ_{(s,p)}=
            \begin{cases}
                0&n>0,\\
                q_p&n=0,\\
                (-1)^n2q_{p+n}&n<0.
            \end{cases}\]
    \end{proposition}

    \begin{proof}
        First, consider the sum $\sum_{r+s=n}(-1)^rq_rQ_{(p,s)}$. If $n<0$ then we have either $r<0$ or $s<0$ in each term, and so the sum is $0$. Similarly, if $n=0$, then the only possible nonzero term is when $r=s=0$, and so we get $(-1)^0q_0Q_{(p,0)}=q_p$.

        So, suppose $n>0$, then we use (\ref{Q_(r,s)}) to expand $Q_{(p,s)}$ and get
        \begin{align*}
             \sum_{r+s=n}(-1)^rq_rQ_{(p,s)}&=\sum_{r+s=n}(-1)^rq_r\left(q_pq_s+2\sum_{i=1}^s(-1)^iq_{p+i}q_{s-i}\right)\\
            &=q_p\sum_{r+s=n}(-1)^rq_rq_s+2\sum_{r+s=n}(-1)^rq_r\sum_{i=1}^s(-1)^iq_{p+i}q_{s-i}.
        \end{align*}
        Note that first sum is $0$ since $n>0$. We may reindex the next sums so that we have
        \[2\sum_{j=0}^{n-1}\sum_{i=1}^{n-j}(-1)^jq_j(-1)^iq_{p+i}q_{n-j-i},\]
        where the outer sum only runs to $j=n-1$ since $q_{n-j-i}=0$ when $j=n$. By swapping the order of the sums, we have
        \[2\sum_{i=1}^n\sum_{j=0}^{n-i}(-1)^jq_j(-1)^iq_{p+i}q_{n-j-i}.\]
        Finally, we can rewrite this to get
        \begin{align*}
            2\sum_{i=1}^{n}&(-1)^iq_{p+i}\sum_{j=0}^{n-i}(-1)^jq_jq_{n-j-i}\\
            &=2\sum_{i=1}^{n}(-1)^iq_{p+i}\sum_{u+v=n-i}(-1)^uq_uq_v\\
            &=2(-1)^nq_{p+n}
        \end{align*}
        since the inner sum is $0$ if $n-i\neq0$, and is $1$ if $i=n$.
        
        Finally, we get the second identity with similar calculations, or by applying Proposition \ref{B_iQ_lambda} to $Q_{(p,s)}$ in the first identity.
    \end{proof}

        

    
    We note that one may prove these identities with generating functions. However, our construction allows for a proof with just basic algebra. Importantly, we can see that it is useful that equation (\ref{Q_(r,s)}) is used as the definition of $Q_{(r,s)}$ for all integers $r$ and $s$, rather than just the case where both are positive.

    \bibliographystyle{alpha}
    \bibliography{references}

\end{document}